\documentclass[preprint,12pt]{elsarticle}

\usepackage{amsthm}
\usepackage{amsmath}
\usepackage{lineno}

\journal{Journal of \LaTeX\ Templates}

\usepackage{amssymb}

\newtheorem{theorem}{Theorem}
\newtheorem{definition}[theorem]{Definition}
\newtheorem{lemma}[theorem]{Lemma}
\newtheorem{proposition}[theorem]{Proposition}
\newtheorem{claim}[theorem]{Claim}

\newcommand{\en}{\mathbb N}

\usepackage{enumerate}
\newtheorem{corollary}[theorem]{Corollary}
\newtheorem{remark}[theorem]{Remark}
\newcommand{\diam}{\operatorname{diam}}
\newcommand{\FTP}{\mathcal{FTP}}
\newcommand{\FOTP}{\mathcal{FOTP}}
\newcommand{\dist}{\operatorname{dist}}

\def\rest{\hskip-2.5pt\restriction}

\begin{document}

\begin{frontmatter}

\title{Haar meager sets, their hulls, and relationship to compact sets}

\author{Martin Dole\v zal\fnref{MD}}
\fntext[MD]{Supported by GA \v{C}R Grant 16-07378S and  RVO: 67985840.}
\address{Institute of Mathematics, Czech Academy of Sciences. \v Zitn\'a 25, 115 67 Praha 1, Czech Republic.}
\ead{dolezal@math.cas.cz}

\author{V\'aclav Vlas\'ak\fnref{VV}}
\fntext[VV]{Supported by the grant GA\v CR P201/15-08218S and he is a junior researcher in the University Center for Mathematical Modeling, Applied Analysis and Computational Mathematics (Math MAC).}
\address{Charles University in Prague, Faculty of Mathematics and Physics, Ke Karlovu 3, 121 16 Praha 2, Czech Republic.}
\ead{vlasakvv@gmail.com}

\begin{abstract}
Let $G$ be an abelian Polish group.
We show that there is a strongly Haar meager set in $G$ without any $F_{\sigma}$ Haar meager hull (and that this still remains true if we replace $F_{\sigma}$ by any other class of the Borel hierarchy). We also prove that there is a coanalytic naively strongly Haar meager set without any Haar meager hull.
Further, we investigate the relationship of the collection of all compact sets to the collection of all Haar meager sets in non-locally compact Polish groups.
\end{abstract}


\end{frontmatter}

\section{Introduction}

The notion of Haar meager sets (in abelian Polish groups) was introduced by Darji \cite{Darji} (and straightforwardly generalized to the case of arbitrary Polish groups in \cite{DRVV}).
Haar meager sets form a topological counterpart to the so called Haar null sets defined by Christensen in \cite{Christensen}.
Let us recall the definitions (and some of its variants).

\begin{definition}\label{definition HN}
	Let $G$ be a Polish group. A set $A\subseteq G$ is said to be
	\begin{itemize}
		\item[(i)] \emph{Haar null} if there are a Borel set $B\subseteq G$ such that $A\subseteq B$, and a Borel probability measure $\mu$ on $G$ such that $\mu(gBh)=0$ for every $g,h\in G$;
		\item[(ii)] \emph{generalised Haar null} if there are a universally measurable set $B\subseteq G$ such that $A\subseteq B$, and a Borel probability measure $\mu$ on $G$ such that $\mu(gBh)=0$ for every $g,h\in G$.
	\end{itemize}
\end{definition}

\begin{definition}\label{definition HM}
	Let $G$ be a Polish group. A set $A\subseteq G$ is said to be
	\begin{itemize}
		\item[(i)] \emph{Haar meager} if there are a Borel set $B\subseteq G$ such that $A\subseteq B$, a compact metric space $K$ and a continuous function $f\colon K\rightarrow G$ such that $f^{-1}(gBh)$ is meager in $K$ for every $g,h\in G$;
		\item[(ii)] \emph{strongly Haar meager} if there are a Borel set $B\subseteq G$ such that $A\subseteq B$, and a compact set $K\subseteq G$ such that $gBh\cap K$ is meager in $K$ for every $g,h\in G$;
		\item[(iii)] \emph{naively Haar meager} if there is a compact metric space $K$ and a continuous function $f\colon K\rightarrow G$ such that $f^{-1}(gAh)$ is meager in $K$ for every $g,h\in G$;
		\item[(iv)] \emph{naively strongly Haar meager} if there is a compact set $K\subseteq G$ such that $gAh\cap K$ is meager in $K$ for every $g,h\in G$.
	\end{itemize}
\end{definition}

Darji proved that in any abelian Polish group $G$, Haar meager sets form a $\sigma$-ideal contained in the $\sigma$-ideal of all meager sets, and that these two $\sigma$-ideals coincide if and only if $G$ is locally compact. These results clearly correlate to the analogous well known results concerning Haar null sets proved by Christensen in \cite{Christensen}.
Other similarities of the $\sigma$-ideals of Haar meager sets and of Haar null sets were investigated in \cite{Jablonska} and \cite{DRVV}. We should note that it is not known whether every (naively) Haar meager set is (naively) strongly Haar meager.

In this paper we study two different topics. First, we investigate whether it is possible to replace the Borel hull $B$ from (i) in Definition \ref{definition HM} by a hull from some other class of sets, e.g. by an $F_{\sigma}$ hull. Next, we investigate the relationship of the collection of all compact sets to the collection of all Haar meager sets in non-locally compact Polish groups.

Let us look at the content of this paper a little closer.
In Chapter \ref{notation} we introduce the notation and recall some facts needed later.
Chapter \ref{F_sigma} is devoted to the descriptive complexity of hulls of Haar meager sets. This chapter is very closely inspired by a paper of Elekes and Vidny{\'a}nszky \cite{EV} (both by the results and by the proof methods). Elekes and Vidny{\'a}nszky answered a question posed by Mycielski \cite[Problem (P1)]{Mycielski} by proving that in any abelian Polish group, there is a Haar null set which cannot be cover by a $G_{\delta}$ Haar null set \cite[Theorem~1.3]{EV}. They also offered a more general statement where $G_{\delta}$ may be replaced by
any other class of the Borel hierarchy \cite[Theorem~1.4]{EV}. Finally, they answered a question asked by Fremlin at his webpage by showing that there are generalised Haar null sets which are not Haar null \cite[Theorem~1.8]{EV}. Comparing the Definitions \ref{definition HN} and \ref{definition HM}, it is reasonable to believe that analogous results could be proved also for Haar meager sets where `Haar null' corresponds to `Haar meager' and the multiplicative class $G_{\delta}$ corresponds to the additive class $F_{\sigma}$. The purpose of Chapter \ref{F_sigma} is showing that this is indeed possible by making only subtle changes in the proofs from \cite{EV}. The main result of Chapter \ref{F_sigma} is Theorem \ref{MainResult} whose special cases are Theorems \ref{MR1} and \ref{MR2}. Note that instead of the family of all universally measurable sets, the topological counterpart to the generalised Haar null sets should consider the family of all sets $B\subseteq G$ such that for every compact metric space $K$ and every continuous function $f\colon K\rightarrow G$, the preimage $f^{-1}(B)\subseteq K$ has the Baire property. However, we do not need to formulate this definition (and its strong variant) since Theorem \ref{MR2} uses a coanalytic naively strongly Haar meager set which is a stronger notion.

The results from Chapter~\ref{compacts} are inspired by the question posed in~\cite[Question~6]{DRVV} which asks whether compact subsets of a non-locally compact Polish group are Haar meager. We provide a sufficient condition for $F_{\sigma}$ subsets of a Polish group to be strongly Haar meager. Then we show if $G$ is either the symmetric group $S_{\infty}$ or any non-locally compact Polish group with a translation invariant metric then all compact subsets of $G$ satisfy this sufficient condition, and thus they are strongly Haar meager. This improves the result by Jab{\l}o{\'n}ska who proved that every compact subset of a non-locally compact abelian Polish group is Haar meager \cite{Jablonska}.

\section{Notation}\label{notation}
For a Polish group $G$ let $\mathcal K(G)$ denote the family of all nonempty compact subsets of $G$ with the Vietoris topology.
If $Z$ is a Polish space then $\mathcal C(Z,G)$ denotes the family of all continuous functions from $Z$ to $G$ with the topology of uniform convergence.

For every $1\le\xi<\omega_1$, let $\Sigma^0_{\xi}$ be the $\xi$th additive Borel class and $\Pi^0_{\xi}$ be the $\xi$th multiplicative Borel class. Let $\Delta^1_1$ denote the class of all Borel sets, $\Sigma_1^1$ the class of all analytic sets and $\Pi_1^1$ the class of all coanalytic sets. Whenever $\Gamma$ is one of these classes then we define its dual class $\check\Gamma$ by $\check\Gamma=\{\sim A\colon A\in\Gamma\}$.

Recall that if $\Gamma$ is a class of sets in Polish spaces and $X$ is a Polish space then a set $U\subseteq 2^{\omega}\times X$ is called \emph{universal for} $\Gamma(X)$ if $U\in\Gamma(2^{\omega}\times X)$ and $\{U_y\colon y\in2^{\omega}\}=\Gamma(X)$.
By \cite[Theorems 22.3 and 26.1]{Kechris}, if $\Gamma$ is one of the classes $\Sigma^0_{\xi}$, $\Pi^0_{\xi}$, $\Sigma^1_1$ or $\Pi^1_1$, and if $X$ is a Polish space then there exists a universal set for $\Gamma(X)$.

Recall that if $X,Y$ are Polish spaces and $\Phi\colon X\rightarrow 2^{2^Y}$ is a function (whose values are collections of subsets of $Y$) then $\Phi$ is said to be \emph{Borel on Borel} if for every Polish space $Z$ and every Borel set $A\subseteq Z\times X\times Y$, the set $\{(z,x)\in Z\times X\colon A_{z,x}\in\Phi(x)\}$ is Borel.

Whenever $s\in\omega^{<\omega}$ is a finite sequence of elements of $\omega$, we write $|s|$ for the length of $s$. 

\section{The descriptive complexity of hulls of Haar meager sets}\label{F_sigma}
First of all, we prove the following simple characterization of Haar meager sets which states that the witnessing compact $K$ from (i) in Definition \ref{definition HM} can be always chosen to be the Cantor set, denoted by $C$.

\begin{proposition}
Let $G$ be a Polish group and let $A$ be a subset of $G$. Then the following assertions are equivalent:
	\begin{itemize}
		\item[(a)] $A$ is Haar meager;
		\item[(b)] there are a Borel set $B\subseteq G$ such that $A\subseteq B$, and $f\in\mathcal C(C,G)$ such that $f^{-1}(gBh)$ is meager in $C$ for every $g,h\in G$.
	\end{itemize}
\end{proposition}

\begin{proof}
The implication (b)$\Rightarrow$(a) is trivial. So suppose that $A$ is Haar meager. Let $B\subseteq G$ be a Borel set such that $A\subseteq B$, let $K$ be a compact metric space, and let $f\colon K\rightarrow G$ be a continuous function such that $f^{-1}(gBh)$ is meager in $K$ for every $g,h\in G$.
It is well known that every nonempty compact metric space is a continuous image of the Cantor set $C$, so there is a continuous surjection $\phi\colon C\rightarrow K$. By \cite[Lemma 2.10]{Darji}, the surjection $\phi$ can be chosen in such a way that the preimages of meager sets (in $K$) are also meager (in $C$). Then the composition $\phi\circ f\colon C\rightarrow G$ also witnesses that $A$ is Haar meager.
\end{proof}

Note that in the further text, we use the Cantor set also in other contexts (for example in Lemma~\ref{coanalytic} where we work with a subset of $2^{\omega}\times 2^{\omega}\times G$). To avoid any misunderstandings, we use the symbol $C$ only when we consider the Cantor set as a (possible) witnessing compact for some set being Haar meager while we use the designation $2^{\omega}$ in all other contexts.

The following two lemmata will be used in the proof of Proposition~\ref{phi}. Lemma~\ref{BorelOnBorel} is a modification of the theorem by Montgomery and Novikov (see \cite[Theorem~16.1]{Kechris}) and is needed for the case $\Gamma=\Delta^1_1$ and $\Lambda=\Sigma^0_{\xi}$ (for some $1\le\xi<\omega_1$) of Proposition~\ref{phi}. Lemma~\ref{coanalytic} is a modification of the theorem by Novikov (see \cite[Theorem 29.22]{Kechris}) and is needed for the complementary case $\Gamma=\Pi^1_1$ and $\Lambda=\Delta^1_1$. These two lemmata form the only new ingredients to the proof methods from \cite{EV}. Indeed, Proposition \ref{phi} is a straightforward analogy of \cite[Theorem~3.1]{EV}, Proposition \ref{PosunutiKompaktu} is an easy consequence of \cite[Proposition~3.5]{EV} and Theorem \ref{MainResult} just puts all the ingredients together (in the same way as \cite[Theorem~4.1]{EV}).

\begin{lemma}\label{BorelOnBorel}
	Let $G$ be a Polish group and let $Z$ be a Polish space.
	Then for every nonempty open set $V\subseteq C$ and for every Borel set $A\subseteq Z\times\mathcal C(C,G)\times G$, the set
	$\{(z,f)\in Z\times\mathcal C(C,G)\colon f^{-1}(A_{z,f})\cap V\textnormal{ is meager in } V\}$ is also Borel.

	In particular, the mapping $f\in\mathcal C(C,G)\mapsto\{B\subseteq G\colon f^{-1}(B)\textnormal{ is meager in }C\}\in2^{2^G}$ is Borel on Borel.
\end{lemma}

\begin{proof}
	Let $\mathcal A$ be the collection of all Borel sets $A\subseteq Z\times\mathcal C(C,G)\times G$ for which the set $\{(z,f)\in Z\times\mathcal C(C,G)\colon f^{-1}(A_{z,f})\cap V\textnormal{ is meager in } V\}$ is Borel for every nonempty open set $V\subseteq C$. We need to show that $\mathcal A$ contains all Borel sets.
	
	First of all, let $A=A_Z\times A_{\mathcal C(C,G)}\times A_G$ where the sets $A_Z\subseteq Z$, $A_{\mathcal C(C,G)}\subseteq\mathcal C(C,G)$ and $A_G\subseteq G$ are open. We want to verify that $A\in\mathcal A$. So let us fix a nonempty open set $V\subseteq C$.
	For every $f\in\mathcal C(C,G)$, the preimage $f^{-1}(A_G)$ is an open subset of $C$. Therefore for every $f\in\mathcal C(C,G)$, the set $f^{-1}(A_G)\cap V$ is meager in $V$ if and only if $f^{-1}(A_G)\cap V=\emptyset$. So we have
	\begin{align*}
		&\{(z,f)\in Z\times\mathcal C(C,G)\colon f^{-1}(A_{z,f})\cap V\textnormal{ is meager in }V\}\\
		=&\left(\left(Z\times\mathcal C(C,G)\right)\setminus\left(A_Z\times A_{\mathcal C(C,G)}\right)\right)\cup\{(z,f)\in A_Z\times A_{\mathcal C(C,G)}\colon f^{-1}(A_G)\cap V=\emptyset\}\;.
	\end{align*}
	To see that this set is Borel, it is clearly enough to verify that the set $\{f\in\mathcal C(C,G)\colon f^{-1}(A_G)\cap V=\emptyset\}$ is Borel, and this is easy (in fact, this set is even closed).

	Next, we show that $\mathcal A$ is closed under countable unions and under taking complements which will finish the proof.
	
	Let $A^n$, $n\in\omega$, be sets from $\mathcal A$ and let $V\subseteq C$ be a nonempty open set. Then we have
	\begin{align*}
		&\left\{(z,f)\in Z\times\mathcal C(C,G)\colon f^{-1}\left(\left(\bigcup\limits_{n\in\omega}A^n\right)_{z,f}\right)\cap V\textnormal{ is meager in }V\right\}\\
		=&\left\{(z,f)\in Z\times\mathcal C(C,G)\colon\bigcup\limits_{n\in\omega} f^{-1}\left(A^n_{z,f}\right)\cap V\textnormal{ is meager in }V\right\}\\
		=&\bigcap\limits_{n\in\omega}\left\{(z,f)\in Z\times\mathcal C(C,G)\colon f^{-1}\left(A^n_{z,f}\right)\cap V\textnormal{ is meager in }V\right\}\;.
	\end{align*}
	This set is the countable intersection of Borel sets, and so it is Borel.
	
	Finally, suppose that $A\in\mathcal A$. Let $V\subseteq C$ be a nonempty open set and let $\mathcal B$ be a countable open basis of $C$. Recall that for every Borel set $M\subseteq C$, the set $M\cap V$ is comeager in $V$ if and only if $M\cap B$ is non-meager in $B$ whenever $B\in\mathcal B$ is such that $\emptyset\neq B\subseteq V$ (see e.g. \cite[Proposition 8.27]{Kechris}). Therefore we have
	\begin{align*}
		&\left\{(z,f)\in Z\times\mathcal C(C,G)\colon f^{-1}\left((\sim A)_{z,f}\right)\cap V\textnormal{ is meager in }V\right\}\\
		=&\left\{(z,f)\in Z\times\mathcal C(C,G)\colon f^{-1}\left(A_{z,f}\right)\cap V\textnormal{ is comeager in }V\right\}\\
		=&\bigcap\limits_{\stackrel{B\in\mathcal B}{\emptyset\neq B\subseteq V}}\left\{(z,f)\in Z\times\mathcal C(C,G)\colon f^{-1}\left(A_{z,f}\right)\cap B\textnormal{ is non-meager in }B\right\}\\
		=&\bigcap\limits_{\stackrel{B\in\mathcal B}{\emptyset\neq B\subseteq V}}\sim\left\{(z,f)\in Z\times\mathcal C(C,G)\colon f^{-1}\left(A_{z,f}\right)\cap B\textnormal{ is meager in }B\right\}\;.
	\end{align*}
	This set is the countable intersection of complements of Borel sets, and so it is Borel.
\end{proof}

In the proof of Lemma \ref{coanalytic} we will use the unfolded **-game from \cite[21.C]{Kechris} which we recall here.
Let $Y$ be a Polish space with a fixed compatible complete metric, and let $\mathcal V$ be a countable basis of nonempty open sets in $Y$. Then for every $E\subseteq Y\times\omega^{\omega}$, let $G^{**}_u(E)$ be the unfolded **-game defined as follows:

\vspace{15pt}
\noindent
\begin{tabular}{lccccc}
	\textrm{I \hskip 40 pt}  &$U_0$&&$U_1$&&\\
	&&&&&$\cdots$\\
	\textrm{II}  &&$(z(0),V_0)$&&$(z(1),V_1)$&\\
\end{tabular}
\vspace{15pt}

\noindent Players I and II play $U_i,V_i\in\mathcal V$ such that $U_0\supseteq V_0\supseteq U_1\supseteq V_1\supseteq\cdots$ and $\textnormal{diam}(U_i),\textnormal{diam}(V_i)<2^{-i}$. Player II also plays $z(i)\in\omega$. Let $y\in Y$ be the unique point such that $\{y\}=\bigcap\limits_{i\in\omega}\overline{U_i}=\bigcap\limits_{i\in\omega}\overline{V_i}$, and let $z=(z(i))_{i\in\omega}\in\omega^{\omega}$. Player II wins the game if and only if $(y,z)\in E$ (the pair $(y,z)$ is called the outcome of the run). Recall that if $E$ is closed then it holds by \cite[Theorem 21.5 and the following comments]{Kechris} that Player II has a winning strategy in the game $G_u^{**}(E)$ if and only if $\textnormal{proj}_YE$ is comeager in $Y$.

\begin{lemma}\label{coanalytic}
Let $G$ be a Polish group, let $Z$ be a Polish space, and let $V\subseteq 2^{\omega}\times 2^{\omega}\times G$ be in $\Sigma^1_1$. Then the sets
	$$\tilde V:=\{(f,x)\in \mathcal C(Z,G)\times 2^{\omega}\colon f^{-1}(V_{x,x})\textnormal{ is non-meager in }Z\}$$
	and
	$$\tilde{\tilde V}:=\{(f,x)\in \mathcal C(Z,G)\times 2^{\omega}\colon f^{-1}(V_{x,x})\textnormal{ is comeager in }Z\}$$
	are also in $\Sigma^1_1$.
\end{lemma}

\begin{proof}
	Let $\mathcal B$ be a countable open basis of $Z$. Whenever $f\in\mathcal C(Z,G)$ and $x\in 2^{\omega}$ then by \cite[Proposition 8.27]{Kechris}, the set $f^{-1}(V_{x,x})$ is non-meager in $Z$ if and only if there is a nonempty set $B\in\mathcal B$ such that $f^{-1}(V_{x,x})\cap B$ is comeager in $B$. So it suffices to show that for every nonempty $B\in\mathcal B$, the set
	$$\tilde V_0:=\{(f,x)\in \mathcal C(Z,G)\times 2^{\omega}\colon f^{-1}(V_{x,x})\cap B\textnormal{ is comeager in }B\}$$
	is in $\Sigma^1_1$ (as we may suppose that $Z\in\mathcal B$). So let us fix a nonempty $B\in\mathcal B$ together with a compatible complete metric on $B$. By \cite[Exercise 14.3]{Kechris} there is a closed set $F\subseteq 2^{\omega}\times 2^{\omega}\times G\times\omega^{\omega}$ with $V=\textnormal{proj}_{2^{\omega}\times 2^{\omega}\times G}(F)$.
	For every $f\in\mathcal C(Z,G)$ and every $x\in 2^{\omega}$, let
	$$F(f,x):=\{(y,z)\in B\times\omega^{\omega}\colon (x,x,f(y),z)\in F\}\;,$$
which is clearly a closed subset of $B\times\omega^{\omega}$. 
	Then for every $f\in\mathcal C(Z,G)$ and every $x\in 2^{\omega}$, Player II has a winning strategy in the game $G_u^{**}(F(f,x))$ if and only if $\textnormal{proj}_BF(f,x)$ is comeager in $B$. But $\textnormal{proj}_BF(f,x)=f^{-1}(V_{x,x})\cap B$, and so we have
	$$\tilde V_0=\{(f,x)\in\mathcal C(Z,G)\times 2^{\omega}\colon\textnormal{Player II has a winning strategy in }G_u^{**}(F(f,x))\}\;.$$
	Let $T$ be the tree of all legal positions in the games $G^{**}_u$ (clearly, the tree $T$ does not depend on the choice of the subset of $B\times\omega^{\omega}$ the game is played with). We can view the tree $T$, as well as every strategy $\sigma\subseteq T$ for Player II, as an element of the space $\textnormal{Tr}$ of all trees on $\en$ (which is a Polish space by \cite[Exercise 4.32]{Kechris}). Then the set $\tilde V_0$ is the projection of the set
	\begin{eqnarray*}
	\tilde V_1&:=&\{(f,x,\sigma)\in\mathcal C(Z,G)\times 2^{\omega}\times\textnormal{Tr}\colon\\
	&&\;\;\sigma\subseteq T\textnormal{ is a winning strategy for Player II in the game }G_u^{**}(F(f,x))\}
	\end{eqnarray*}
to the first two coordinates, and so it is enough to verify that $\tilde V_1$ is analytic.

For every legal position $p\in T$ in the games $G^{**}_u$ we define a (closed) subset $Q(p)$ of $B\times\omega^{\omega}$ as the set of all possible outcomes of those runs of the games $G^{**}_u$ which start according to the position $p$. This means that if $p=\left(U_0,(z(0),V_0),U_1,\ldots,(z(n),V_n)\right)$ is of even length then we put $$Q(p)=\overline{V_n}\times\{z\in\omega^{\omega}\colon (z_0,\ldots,z_n)\textnormal{ is an initial segment of }z\}\;.$$
And if $p=\left(U_0,(z(0),V_0),U_1,\ldots,(z(n),V_n),U_{n+1}\right)$ is of odd length then we put $$Q(p)=\overline{U_{n+1}}\times\{z\in\omega^{\omega}\colon (z_0,\ldots,z_n)\textnormal{ is an initial segment of }z\}\;.$$
The proof of the following claim is already included in the proof of \cite[Theorem 29.22]{Kechris} but for completeness sake we briefly repeat it here.
	
	\begin{claim}
		Let $\sigma\subseteq T$ be a strategy for Player II in the games $G^{**}_u$. Let $f\in\mathcal C(Z,G)$ and $x\in2^{\omega}$. Then $\sigma$ is a winning strategy for Player II in the game $G_u^{**}(F(f,x))$ if and only if for every $p\in\sigma$ we have $Q(p)\cap F(f,x)\neq\emptyset$.	
	\end{claim}
	
	\begin{proof}
Let us fix a compatible complete metric on $\omega^{\omega}$, and consider the product metric on $B\times\omega^{\omega}$ (recall that we have fixed a compatible complete metric on $B$, so the product metric is also complete).		
For every infinite branch $b\in[T]$, the sequence $\left(Q(b|n)\right)_{n\in\omega}$ is a decreasing (with respect to inclusion) sequence of closed subsets of $B\times\omega^{\omega}$ whose diameters tend to 0. So for every $b\in[T]$, the intersection $\bigcap_{n\in\omega}Q(b|n)$ is a singleton whose only element $q(b)$ is the outcome of the run in the game $G_u^{**}(F(f,x))$ corresponding to $b$. Now the strategy $\sigma$ is winning for Player II in the game $G_u^{**}(F(f,x))$ if and only if $q(b)\in F(f,x)$ for every $b\in[\sigma]$, and this clearly holds if and only if $Q(b|n)\cap F(f,x)\neq\emptyset$ for every $b\in[\sigma]$ and $n\in\omega$. This finishes the proof of the claim.		
	\end{proof}
	
	By the previous claim we have
	\begin{eqnarray*}
	\tilde V_1&=&\{(f,x,\sigma)\in\mathcal C(Z,G)\times 2^{\omega}\times\textnormal{Tr}\colon\\
	&&\;\;\sigma\subseteq T\textnormal{ is a strategy for Player II and }\forall_{p\in\sigma}\;Q(p)\cap F(f,x)\neq\emptyset\}\\
	&=&\{(f,x,\sigma)\in\mathcal C(Z,G)\times 2^{\omega}\times\textnormal{Tr}\colon\\
	&&\;\;\sigma\subseteq T\textnormal{ is a strategy for Player II and }\forall_{p\in\sigma}\;\exists_{(y,z)\in Q(p)}\; (x,x,f(y),z)\in F\},
	\end{eqnarray*}
which is clearly an analytic set.
\end{proof}

\begin{proposition}\label{phi}
Let either $\Gamma=\Delta^1_1$ and $\Lambda=\Sigma^0_{\xi}$ for some $1\le\xi<\omega_1$, or let $\Gamma=\Pi^1_1$ and $\Lambda=\Delta^1_1$.
Let $G$ be a Polish group.
	Then there exists a partial function $\phi\colon\mathcal C(C,G)\times2^{\omega}\rightarrow G$ such that $\textnormal{graph}(\phi)\in\Gamma$, and such that for every $f\in\mathcal C(C,G)$ it holds
	\begin{itemize}
		\item[(i)] $\forall x\in 2^{\omega}\colon (f,x)\in\textnormal{dom}(\phi)\Rightarrow \phi(f,x)\in f(C)$,
		\item[(ii)] $\forall S\in\Lambda(2^{\omega}\times G)\colon\textnormal{graph}(\phi_f)\subseteq S\Rightarrow$ there is $x\in2^{\omega}$ such that $f^{-1}(S_x)$ is non-meager in $C$.
	\end{itemize}
\end{proposition}

\begin{proof}
	Let $U\subseteq2^{\omega}\times2^{\omega}\times G$ be a set which is
	\begin{itemize}
		\item universal for $\Pi^0_{\xi}(2^{\omega}\times G)$ if [$\Gamma=\Delta^1_1$ and $\Lambda=\Sigma^0_{\xi}$ for some $1\le\xi<\omega_1$];
		\item universal for $\Pi^1_1(2^{\omega}\times G)$ if [$\Gamma=\Pi^1_1$ and $\Lambda=\Delta^1_1$].
	\end{itemize}
	We put
	$$\tilde U:=\{(f,x,g)\in\mathcal C(C,G)\times2^{\omega}\times G\colon f^{-1}(U_{x,x})\textnormal{ is non-meager in }C, g\in U_{x,x}\cap f(C)\}.$$
	
	\begin{claim}
		The set $\tilde U$ is from the class $\Gamma$.
	\end{claim}
	
	\begin{proof}
	It is easy to see that the set $\{(f,g)\in\mathcal C(C,G)\times G\colon g\in f(C)\}$ is closed.
	The set $U$ is in $\Gamma$ by its definition, and so the set $\{(x,g)\in2^{\omega}\times G\colon g\in U_{x,x}\}$ is also in $\Gamma$.
	Putting these two facts together we get that the set
$$\{(f,x,g)\in\mathcal C(C,G)\times2^{\omega}\times G\colon g\in U_{x,x}\cap f(C)\}$$
is in $\Gamma$. 
	So it remains to show that the set
$$\{(f,x)\in\mathcal C(C,G)\times2^{\omega}\colon f^{-1}(U_{x,x})\textnormal{ is non-meager in }C\}$$
is in $\Gamma$ as well.

In the case of $\Gamma=\Delta^1_1$ it suffices to use (the first part of) Lemma \ref{BorelOnBorel} on the Borel set
$$A:=\{(x,f,g)\in 2^{\omega}\times\mathcal C(C,G)\times G\colon (x,x,g)\in U\}\;,$$
as for every $x\in2^{\omega}$ and every $f\in\mathcal C(C,G)$ it holds $f^{-1}(A_{x,f})=f^{-1}(U_{x,x})$.

In the case of $\Gamma=\Pi^1_1$ let $\mathcal B$ be a countable basis of nonempty open sets in $C$. Then (using \cite[Proposition 8.27]{Kechris} in the first equality) we have
	\begin{eqnarray*}
	&\{(f,x)\in\mathcal C(C,G)\times2^{\omega}\colon f^{-1}(U_{x,x})\textnormal{ is non-meager in }C\}\\
	=&\bigcup\limits_{B\in\mathcal B}\{(f,x)\in\mathcal C(C,G)\times2^{\omega}\colon f^{-1}(U_{x,x})\cap B\textnormal{ is comeager in }B\}\\
	=&\sim\bigcap\limits_{B\in\mathcal B}\{(f,x)\in\mathcal C(C,G)\times2^{\omega}\colon f^{-1}(\sim U_{x,x})\cap B\textnormal{ is non-meager in }B\}.
	\end{eqnarray*}
Therefore it is enough to show that for every $B\in\mathcal B$, the set
\begin{eqnarray*}
&\{(f,x)\in\mathcal C(C,G)\times2^{\omega}\colon f^{-1}(\sim U_{x,x})\cap B\textnormal{ is non-meager in }B\}\\
=&\{(f,x)\in\mathcal C(C,G)\times2^{\omega}\colon \left(f\rest_{B}\right)^{-1}((\sim U)_{x,x})\textnormal{ is non-meager in }B\}
\end{eqnarray*}
is in $\Sigma^1_1$. But this easily follows from Lemma \ref{coanalytic} used on the $\Sigma^1_1$ set $V:=\sim U$.
	\end{proof}
	
Whenever $f\in\mathcal C(C,G)$ and $x\in2^{\omega}$ are such that $f^{-1}(U_{x,x})$ is non-meager in $C$ then $\tilde U_{f,x}=U_{x,x}\cap f(C)$. On the other hand, if $f\in\mathcal C(C,G)$ and $x\in2^{\omega}$ are such that $f^{-1}(U_{x,x})$ is meager in $C$ then $\tilde U_{f,x}=\emptyset$. It follows that for every $(f,x)\in\textnormal{proj}_{\mathcal C(C,G)\times 2^{\omega}}\tilde U$ we have $f^{-1}(\tilde U_{f,x})$ is non-meager in $C$.
This enables us to define $\phi$ as the uniformization of the set $\tilde U\subseteq(\mathcal C(C,G)\times 2^{\omega})\times G$ such that $\textnormal{graph}(\phi)\in\Gamma$. Indeed, in the case of $\Gamma=\Delta^1_1$ the existence of such a uniformization follows by (the second part of) Lemma~\ref{BorelOnBorel} and by [Kechris, Theorem~18.6]. And in the case of $\Gamma=\Pi^1_1$ the existence follows by the unformization property of the class $\Pi^1_1$ (see \cite[Theorem~36.14]{Kechris}).

We have $\textnormal{dom}(\phi)=\{(f,x)\in\mathcal C(C,G)\times2^{\omega}\colon f^{-1}(U_{x,x})\textnormal{ is non-meager in }C\}$, and $\phi(f,x)\in U_{x,x}\cap f(C)$ for every $(f,x)\in\textnormal{dom}(\phi)$. In particular, the partial function $\phi$ satisfies (i) for every $f\in\mathcal C(C,G)$.
	
Next we fix $f\in\mathcal C(C,G)$, and we verify (ii). Assume towards a contradiction that there is $S\in\Lambda(2^{\omega}\times G)$ such that $\textnormal{graph}(\phi_f)\subseteq S$ and for every $x\in2^{\omega}$ the set $f^{-1}(S_x)$ is meager in $C$. By the universality of $U$, there is $x\in2^{\omega}$ such that $U_x=(2^{\omega}\times G)\setminus S$. Then $f^{-1}(U_{x,x})=f^{-1}(G\setminus S_x)=C\setminus f^{-1}(S_x)$ is non-meager in $C$ (even comeager in $C$), and so $(f,x)\in\textnormal{dom}(\phi)$ and $\phi(f,x)\in U_{x,x}\cap f(C)\subseteq U_{x,x}=G\setminus S_x$. On the other hand, the inclusion $\textnormal{graph}(\phi_f)\subseteq S$ implies that $\phi(f,x)\in S_x$ which is a contradiction.
\end{proof}

\begin{proposition}\label{PosunutiKompaktu}
Let $G$ be a non-locally compact abelian Polish group and let $K\in\mathcal K(G)$ be fixed. Then there is a Borel map $u\colon\mathcal C(C,G)\times2^{\omega}\times2^{\omega}\rightarrow G$ such that
	\begin{itemize}
		\item[(A)] if $(f,x,y)\neq(f',x',y')$ are elements of $\mathcal C(C,G)\times2^{\omega}\times2^{\omega}$ then
		$$(f(C)-K+u(f,x,y))\cap(f'(C)-K+u(f',x',y'))=\emptyset\;;$$
		\item[(B)] for every $f\in\mathcal C(C,G)$ and $y\in2^{\omega}$ the map $u(f,\cdot,y)$ is continuous.
	\end{itemize}
\end{proposition}

\begin{proof}
	By \cite[Proposition 3.5]{EV} there is a Borel map $t\colon\mathcal K(G)\times2^{\omega}\times2^{\omega}\rightarrow G$ such that
	\begin{itemize}
		\item[($\tilde{\textnormal A}$)] if $(L,x,y)\neq(L',x',y')$ are elements of $\mathcal K(G)\times2^{\omega}\times2^{\omega}$ then
		$$(L-K+t(L,x,y))\cap(L'-K+t(L',x',y'))=\emptyset\;;$$
		\item[($\tilde{\textnormal B}$)] for every $L\in\mathcal K(G)$ and $y\in2^{\omega}$ the map $t(L,\cdot,y)$ is continuous.
	\end{itemize}
	Let us fix a Borel bijection $\theta\colon \mathcal C(C,G)\times2^{\omega}\rightarrow 2^{\omega}$.
	For $(f,x,y)\in\mathcal C(C,G)\times2^{\omega}\times2^{\omega}$ we define $u(f,x,y):=t(f(C),x,\theta(f,y))$. This map clearly works.
\end{proof}

The following theorem is the main result of this chapter.

\begin{theorem}\label{MainResult}
Let either $\Gamma=\Delta^1_1$ and $\Lambda=\Sigma^0_{\xi}$ for some $1\le\xi<\omega_1$, or let $\Gamma=\Pi^1_1$ and $\Lambda=\Delta^1_1$.
	Let $G$ be a non-locally compact abelian Polish group. Then there is a strongly Haar meager (naively strongly Haar meager in the case of $\Gamma=\Pi^1_1$) set $E\in\Gamma(G)$ without any Haar meager hull in $\Lambda(G)$.
\end{theorem}

\begin{proof}
	We fix a Borel bijection $c\colon\mathcal C(C,G)\rightarrow2^{\omega}$ and a perfect compact set $K\in\mathcal K(G)$ such that $0\in K$. Let $\phi\colon\mathcal C(C,G)\times2^{\omega}\rightarrow G$ be the partial function from Proposition \ref{phi}, and let $u\colon\mathcal C(C,G)\times2^{\omega}\times2^{\omega}\rightarrow G$ be the map from Proposition \ref{PosunutiKompaktu}. Let $\Psi\colon\mathcal C(C,G)\times2^{\omega}\times G\rightarrow G$ be defined by $\Psi(f,x,g)=g+u(f,x,c(f))$. We put $E:=\Psi(\textnormal{graph}(\phi))$.
	
	We first prove that $E$ is in the class $\Gamma$. The map $\Psi$ is clearly Borel. Let $D:=\{(f,x,g)\in\mathcal C(C,G)\times2^{\omega}\times G\colon g\in f(C)\}$, then $D$ is a Borel set (even closed). Let $(f,x,g)$ and $(f',x',g')$ be two distinct elements of $D$.
	If $(f,x)\neq(f',x')$ then by condition (A) from Proposition \ref{PosunutiKompaktu} we have $\Psi(f,x,g)\neq\Psi(f',x',g')$ (as $g\in f(C)$, $g'\in f'(C)$ and $0\in K$).
	And if $(f,x)=(f',x')$ and $g\neq g'$ then it clearly holds $\Psi(f,x,g)\neq\Psi(f',x',g')$ as well.
	It follows that $\Psi$ is injective on the set $D$, and so it is a Borel isomorphism of the Borel sets $D$ and $\Psi(D)$. By condition (i) from Proposition \ref{phi} we have $\textnormal{graph}(\phi)\subseteq D$. Moreover $\textnormal{graph}(\phi)$ is in the class $\Gamma$ (again by Proposition \ref{phi}), and so it follows that $E=\Psi(\textnormal{graph}(\phi))$ is also in the class $\Gamma$.
	
	Next we show that $E$ is strongly Haar meager (naively strongly Haar meager in the case of $\Gamma=\Pi^1_1$), witnessed by the compact $K$. We have
	\begin{align*}
	E&=\{\Psi(f,x,\phi(f,x))\colon(f,x)\in\textnormal{dom}(\phi)\}\\
	&=\{\phi(f,x)+u(f,x,c(f))\colon(f,x)\in\textnormal{dom}(\phi)\}\\
	&\subseteq\bigcup\limits_{(f,x)\in\textnormal{dom}(\phi)}\left(f(C)+u(f,x,c(f))\right)
	\end{align*}
	where the sets from the last union are pairwise disjoint by condition (A) from Proposition \ref{PosunutiKompaktu}. It follows that for every $(f,x)\in\textnormal{dom}(\phi)$, the intersection of $E$ and $f(C)+u(f,x,c(f))$ is a singleton $\{\phi(f,x)+u(f,x,c(f))\}$. Let us fix $g\in G$. Then $K-g$ intersects at most one of the sets $f(C)+u(f,x,c(f))$ (otherwise $-g$ would be in the intersection of two distinct sets of the form $f(C)-K+u(f,x,c(f))$ which would contradict Proposition \ref{PosunutiKompaktu}). Therefore $K-g$ intersects $E$ in at most one point, and so $E+g$ intersects $K$ in at most one point. Thus $E+g$ is meager in $K$, as $K$ is perfect.
	
	Finally, we show that there is no Haar meager hull of $E$ in $\Lambda(G)$. Suppose for a contradiction that $H$ is such a hull and let $f\colon C\rightarrow G$ be the witnessing continuous function. By condition (B) from Proposition \ref{PosunutiKompaktu} we easily have that the map $\Psi_f=\Psi(f,\cdot,\cdot)$ is continuous. Thus $S:=\Psi_f^{-1}(H)\in\Lambda(2^{\omega}\times G)$. For every $(x,g)\in\textnormal{graph}(\phi_f)$, we have $\Psi_f(x,g)=\phi(f,x)+u(f,x,c(f))\in E\subseteq H$, and so $\textnormal{graph}(\phi_f)\subseteq S$. By condition (ii) from Proposition \ref{phi} there is $x\in2^{\omega}$ such that $f^{-1}(S_x)$ is non-meager in $C$. But we also have $S_x+u(f,x,c(f))=\Psi_{f,x}(S_x)\subseteq\Psi_f(S)\subseteq H$, and so $f^{-1}(H-u(f,x,c(f)))\supseteq f^{-1}(S_x)$ is non-meager in $C$. This is a contradiction with the fact that $f$ is a witnessing function for $H$ being Haar meager.
\end{proof}

As it was already noted in \cite{EV} in case of Haar null sets, Theorem \ref{MainResult} has the following easy consequence concerning the additivity of the $\sigma$-ideal of all (strongly) Haar meager sets.

\begin{corollary}
Let $G$ be a non-locally compact abelian Polish group. Then the additivity of the $\sigma$-ideal of all (strongly) Haar meager subsets of $G$ is $\omega_1$.
\end{corollary}

\begin{proof}
For every $1\le\xi<\omega_1$, let $E_{\xi}\subseteq G$ be a (strongly) Haar meager set without any Haar meager hull in $\Sigma^0_{\xi}$. Then $\bigcup\limits_{1\le\xi<\omega_1}E_{\xi}$ has no Borel Haar meager hull, and so it is not (strongly) Haar meager.
\end{proof}

The next two theorems are just special cases of Theorem \ref{MainResult} but we write them down explicitly as they were the main motivation for this chapter. Theorem \ref{MR1} answers in negative a question posed in \cite[Question~9]{DRVV}. Theorem \ref{MR2} is in a sharp contrast with \cite[Proposition~8]{DRVV} which states that every $\Sigma^1_1$ naively Haar meager set is Haar meager.

\begin{theorem}\label{MR1}
Let $G$ be a non-locally compact abelian Polish group. Then there is a strongly Haar meager set without any $F_{\sigma}$ Haar meager hull.
\end{theorem}

\begin{theorem}\label{MR2}
Let $G$ be a non-locally compact abelian Polish group. Then there is a coanalytic naively strongly Haar meager set without any Haar meager hull.
\end{theorem}

\section{Relationship between Haar meager sets and compact sets}\label{compacts}




In the following definitions we introduce two conditions, both of which are sufficient for $F_{\sigma}$ sets to be strongly Haar meager. Note that a related result is also proved in \cite[Proposition 5.9]{BaKaRa} where it is shown that a closed set is Haar meager if and only if it is not prethick.

\begin{definition}

Let $G$ be a Polish group. We say that a set $A\subseteq G$ satisfies the \emph{finite translation property} $\FTP$ if for every open set $\emptyset\neq U\subseteq G$ there exists a finite set $M\subseteq U$ such that for every $g,h\in G$ we have $gMh\nsubseteq A$.

\end{definition}

\begin{definition}

Let $G$ be a Polish group. We say that a set $A\subseteq G$ satisfies the \emph{finite open translation property} $\FOTP$ if for every open set $\emptyset\neq U\subseteq G$ there exists a finite collection $\mathcal{M}$ of open nonempty subsets of $U$ such that for every $g,h\in G$ there exists $V\in\mathcal{M}$ with $gVh\cap A=\emptyset$.

\end{definition}

\begin{remark}
Let $G$ be a Polish group, and let $A\subseteq G$. Then clearly

\begin{itemize}

\item{} $A$ satisfies $\FOTP$ if and only if $\overline{A}$ satisfy $\FOTP$,
\item{} if $A$ satisfies $\FOTP$ then $A$ satisfy $\FTP$.

\end{itemize}

\end{remark}


\begin{theorem}\label{FTP}

Let $G$ be a Polish group. Suppose that $A\subseteq G$ is an $F_{\sigma}$ set which satisfies $\FTP$. Then $A$ is strongly Haar meager.

\end{theorem}

\begin{proof}

We fix a compatible complete metric on $G$ such that $\diam(G)\leq 1$. We will inductively construct a nonempty finitely branching pruned tree $T\subseteq\omega^{<\omega}$, together with nonempty open sets $U_t\subseteq G$ and points $x_t\in G$ for $t\in T$ such that

\begin{itemize}

\item[(i)]$\diam(U_t)\leq 2^{-|t|}$, $t\in T$,
\item[(ii)]if $s,t\in T$ are such that $s$ is an initial segment of $t$ then $\overline{U_t}\subseteq U_s$,
\item[(iii)]if $s,t\in T$ are such that $|s|=|t|$ and $s\neq t$, then $\overline{U_s}\cap \overline{U_t}=\emptyset$,
\item[(iv)]$x_t\in U_t$, $t\in T$,
\item[(v)]for every $t\in T$ we have $t\hat{\ }0\in T$ and $x_{t\hat{\ }0}=x_t$,
\item[(vi)]for every $t\in T$ and $g,h\in G$ we have $\{x_{t\hat{\ }i}\in G\colon i\in\omega\textnormal{ such that }t\hat{\ }i\in T\}\nsubseteq gAh$.

\end{itemize}

We start the construction by setting $U_{\emptyset}=G$, and by choosing $x_{\emptyset}$ to be an arbitrary element of $G$. Now, suppose that for some $n\in\omega$, we have already constructed all elements of the tree $T$ whose length is at most $n$, as well as the sets $U_t$ and points $x_t$ for $t\in T$, $|t|\le n$. Let $t\in T$ with $|t|=n$ be fixed. Since $A$ satisfies $\FTP$ there exists a finite set $M\subseteq U_t$ such that for all $g,h\in G$ we have $gMh\nsubseteq A$. Let $x_{t\hat{\ }0},\ldots,x_{t\hat{\ }k}$ be pairwise distinct elements of $G$ such that $x_{t\hat{\ }0}=x_t$ and $\{x_{t\hat{\ }0},\ldots,x_{t\hat{\ }k}\}=M\cup\{x_t\}$. We define $t\hat{\ }i\in T$ if and only if $i\le k$. We find $U_{t\hat{\ }i}$ for $i\le k$ such that the conditions (i)-(iv) are satisfied. This construction clearly works.

Let $[T]=\{t\in\omega^{\omega}\colon t\rest_k\in T\textnormal{ for every }k\in\omega\}$ be the set of all infinite branches of the tree $T$. We put 
$$L=\bigcup_{t\in [T]}\bigcap_{k\in\omega}U_{t\ \rest_k}.$$
Then $L\subseteq G$ is a compact set by (i)-(iii). We show that $L$ is the witnessing compact for the fact that $A$ is strongly Haar meager. Since $A$ is $F_{\sigma}$, it suffices to show that the relative (in $L$) interior of each translation of $A$ is empty. By (ii), (iv) and (v) we have that $\{x_t\colon t\in X\}\subseteq L$. By this fact together with (ii), (iv) and (vi) we obtain that for every $t\in X$ and every $g,h\in G$ we have $U_t\cap L\nsubseteq gAh$. This finishes the proof as the sets $U_t\cap L$, $t\in T$, clearly form a basis of relatively open sets in $L$.
\end{proof}

\begin{corollary}

Let $G$ be a Polish group. Suppose that $A\subseteq G$ satisfies $FOTP$. Then $A$ is strongly Haar meager.

\end{corollary}

\begin{proof}

Since $A$ satisfies $\FOTP$, its closure $\overline{A}$ also satisfies $\FOTP$. Thus $\overline{A}$ satisfies $\FTP$. By Theorem~\ref{FTP} we have that $\overline{A}$ is strongly Haar meager, so $A$ is strongly Haar meager as well.
\end{proof}

Theorems \ref{S_infty} and \ref{InvariantMetric} provide examples of Polish groups in which all compact sets satisfy $\FTP$ (even $\FOTP$ in the latter case).

\begin{theorem}\label{S_infty}

Every compact subset of $S_{\infty}$ satisfies $\FTP$.

\end{theorem}

\begin{proof}
Let $K\subseteq S_{\infty}$ be a compact set.
Let $s\in\omega^{\omega}$ be a nondecreasing sequence such that for all $n\in\omega$ we have
$$\#\{x(n)\in\omega\colon x\in K\}\leq s(n).$$
Let $\emptyset\neq U\subseteq G$ be open. Let us fix $z\in U$ and find $k\in\omega$ such that 
$$\{y\in S_{\infty}\colon z\rest_k\textnormal{ is an initial segment of }y\}\subseteq U.$$
Clearly, we can find $M=\{x_0,\dots,x_{s(k)}\}\subseteq S_{\infty}$ such that $z\rest_k$ is an initial segment of $x_i$ for every $i\le s(k)$, and such that $x_i(n)\neq x_j(n)$ for every $n\geq k$ and $i\neq j$. Then $M\subseteq U$. Let us fix $g,h\in S_{\infty}$, and let us fix some $i\geq k$ with $h^{-1}(i)\leq k$. Then we have
\begin{eqnarray}\nonumber
\#\{y(h^{-1}(i))\colon y\in gMh\}&=&\#\{gyh(h^{-1}(i))\colon y\in M\}\\
&=&\#\{gy(i);\ y\in M\}\nonumber\\
&=&\#\{y(i)\colon y\in M\}\nonumber\\
&=&s(k)+1\nonumber\\
&>&s(h^{-1}(i))\nonumber\\
&\geq&\#\{y(h^{-1}(i))\colon y\in K\}.\nonumber
\end{eqnarray}
Thus $gMh\nsubseteq K$.
\end{proof}

\begin{corollary}

All compact subsets of $S_{\infty}$ are strongly Haar meager.

\end{corollary}

\begin{theorem}\label{InvariantMetric}

Let $G$ be a non-locally compact Polish group which admits a translation invariant metric. Then every compact subset of $G$ satisfies $\FOTP$.

\end{theorem}

\begin{proof}
Let us fix a translation invariant metric on $G$.
By \cite[Coralloary 1.2.2]{BeKe} this metric is complete.
Let $K\subseteq G$ be a compact set, and let $\emptyset\neq U\subseteq G$ be an arbitrary open set. Since $G$ is not locally compact, the set $U$ is not totally bounded. Thus, there exists $\varepsilon>0$ and an infinite set $N=\{x_i\colon i\in\omega\}\subseteq U$ such that for every $i\in\omega$ we have $\dist(\{x_i\},N\setminus\{x_i\})\geq 4\varepsilon$. By the compactness of $K$ there is $n\in\omega$ and a finite set $\{y_0,\dots,y_n\}\subseteq K$ such that
$$K\subseteq\bigcup\limits_{i=0}^nB(y_i,\varepsilon).$$
We set 
$$\mathcal{M}=\{B(x_i,\varepsilon)\cap U\colon i=0,\ldots,n+1\}.$$
Using the fact that the fixed metric is translation invariant, it is easy to see that for every $g,h\in G$ and every $i\leq n$ we have
$$\#\{B\in\mathcal{M};\ gBh\cap B(y_i,\varepsilon)\}\leq 1.$$
Thus, for every $g,h\in G$ there exists $B\in\mathcal{M}$ such that
$$gBh\cap K\subseteq gBh\cap\bigcup_{i=0}^n B(y_i,\epsilon)=\emptyset.$$
This shows that $K$ satisfies $\FOTP$.
\end{proof}

The following corollary improves the result by Jab{\l}o{\'n}ska who proved that every compact subset of a non-locally compact abelian Polish group is Haar meager \cite[Corollary 1]{Jablonska}.

\begin{corollary}

Let $G$ be a non-locally compact Polish group which admits a translation invariant metric. Then all compact subsets of $G$ are strongly Haar meager.

\end{corollary}

The following theorem is a weaker reformulation of the result proved by Matou\v skov\'a in \cite[Theorem 2.3]{Matouskova}.

\begin{theorem}\label{EMat}

Let $X$ be a separable reflexive Banach space. Then every bounded, convex and nowhere dense subset of $X$ satisfies $\FTP$.

\end{theorem}

\begin{corollary}

Let $X$ be a separable reflexive Banach space. Then every convex and nowhere dense subset of $X$ is Haar meager.

\end{corollary}

\begin{proof}
Each convex and nowhere dense subset of $X$ is the countable union of bounded, convex and nowhere dense sets. Thus the assertion follows from the fact that Haar meager sets form a $\sigma$-ideal, together with Theorems \ref{FTP} and \ref{EMat}.
\end{proof}


\bibliography{Ref}

\begin{thebibliography}{10}
\expandafter\ifx\csname url\endcsname\relax
  \def\url#1{\texttt{#1}}\fi
\expandafter\ifx\csname urlprefix\endcsname\relax\def\urlprefix{URL }\fi
\expandafter\ifx\csname href\endcsname\relax
  \def\href#1#2{#2} \def\path#1{#1}\fi

\bibitem{Darji}
U.~B. Darji, \href{http://dx.doi.org/10.1016/j.topol.2013.07.034}{On {H}aar
  meager sets}, Topology Appl. 160~(18) (2013) 2396--2400.
\newblock \href {http://dx.doi.org/10.1016/j.topol.2013.07.034}
  {\path{doi:10.1016/j.topol.2013.07.034}}.
\newline\urlprefix\url{http://dx.doi.org/10.1016/j.topol.2013.07.034}

\bibitem{DRVV}
M.~Dole\v{z}al, M.~Rmoutil, B.~Vejnar, V.~Vlas\'ak, Haar meager sets
  revisitedTo appear in \emph{J. Math. Anal. Appl.}

\bibitem{Christensen}
J.~P.~R. Christensen, On sets of {H}aar measure zero in abelian {P}olish
  groups, Israel J. Math. 13 (1972) 255--260 (1973).

\bibitem{Jablonska}
E.~Jab{\l}o{\'n}ska, \href{http://dx.doi.org/10.1016/j.jmaa.2014.08.005}{Some
  analogies between {H}aar meager sets and {H}aar null sets in abelian {P}olish
  groups}, J. Math. Anal. Appl. 421~(2) (2015) 1479--1486.
\newblock \href {http://dx.doi.org/10.1016/j.jmaa.2014.08.005}
  {\path{doi:10.1016/j.jmaa.2014.08.005}}.
\newline\urlprefix\url{http://dx.doi.org/10.1016/j.jmaa.2014.08.005}

\bibitem{EV}
M.~Elekes, Z.~Vidny{\'a}nszky,
  \href{http://dx.doi.org/10.1007/s11856-015-1216-2}{Haar null sets without
  {$G_\delta$} hulls}, Israel J. Math. 209~(1) (2015) 199--214.
\newblock \href {http://dx.doi.org/10.1007/s11856-015-1216-2}
  {\path{doi:10.1007/s11856-015-1216-2}}.
\newline\urlprefix\url{http://dx.doi.org/10.1007/s11856-015-1216-2}

\bibitem{Mycielski}
J.~Mycielski, Some unsolved problems on the prevalence of ergodicity,
  instability, and algebraic independence, Ulam Quart. 1~(3) (1992) 30ff.,
  approx.\ 8 pp.\ (electronic only).

\bibitem{Kechris}
A.~S. Kechris, \href{http://dx.doi.org/10.1007/978-1-4612-4190-4}{Classical
  descriptive set theory}, Vol. 156 of Graduate Texts in Mathematics,
  Springer-Verlag, New York, 1995.
\newblock \href {http://dx.doi.org/10.1007/978-1-4612-4190-4}
  {\path{doi:10.1007/978-1-4612-4190-4}}.
\newline\urlprefix\url{http://dx.doi.org/10.1007/978-1-4612-4190-4}

\bibitem{BaKaRa}
T.~Banakh, L.~Karchevska, A.~Ravsky, The closed {S}teinhaus properties of
  $\sigma$-ideals on topological groupsPreprint.

\bibitem{BeKe}
H.~Becker, A.~S. Kechris, \href{http://dx.doi.org/10.1017/CBO9780511735264}{The
  descriptive set theory of {P}olish group actions}, Vol. 232 of London
  Mathematical Society Lecture Note Series, Cambridge University Press,
  Cambridge, 1996.
\newblock \href {http://dx.doi.org/10.1017/CBO9780511735264}
  {\path{doi:10.1017/CBO9780511735264}}.
\newline\urlprefix\url{http://dx.doi.org/10.1017/CBO9780511735264}

\bibitem{Matouskova}
E.~Matou{\v{s}}kov{\'a},
  \href{http://dx.doi.org/10.1112/S0024609301008372}{Translating finite sets
  into convex sets}, Bull. London Math. Soc. 33~(6) (2001) 711--714.
\newblock \href {http://dx.doi.org/10.1112/S0024609301008372}
  {\path{doi:10.1112/S0024609301008372}}.
\newline\urlprefix\url{http://dx.doi.org/10.1112/S0024609301008372}

\end{thebibliography}

\end{document}